\documentclass[12pt]{article}
\usepackage{graphicx} 
\usepackage{xcolor}

\usepackage[english]{babel}
\usepackage{amsfonts}
\usepackage{mathtools}
\usepackage[left=2cm,right=2cm,
top=2cm,bottom=2cm,bindingoffset=0cm]{geometry}

\usepackage{amssymb,amsmath,amsthm,amsfonts,bbm}
\usepackage[pdftex,colorlinks=true,linkcolor=blue,citecolor=blue,urlcolor=red,unicode=true,hyperfootnotes=false,bookmarksnumbered]{hyperref}

\usepackage{enumerate}

\theoremstyle{plain}
\newtheorem{theorem}{Theorem}[section]
\newtheorem{lemma}[theorem]{Lemma}
\newtheorem{claim}[theorem]{Claim}

\theoremstyle{definition} 

\newtheorem*{remark}{Remark}

\makeatletter

\author{Noga Alon\footnote{Department of Mathematics, Princeton University, USA; nalon@math.princeton.edu. Research supported in part by NSF grant DMS-2154082}, \, 
Itai Benjamini\footnote{The Weizmann Institute of Science, Israel; itai.benjamini@weizmann.ac.il}, \,
Georgii Zakharov\footnote{The University of Oxford, UK; georgii.zakharov@exeter.ox.ac.uk}, \, 
Maksim Zhukovskii\footnote{School of Computer Science, The University of Sheffield, UK; m.zhukovskii@sheffield.ac.uk}}

\date{}

\title{Sums along the edges of bounded degree graphs}

\begin{document}

\maketitle

\begin{abstract}
Let $G$ be a graph on $n$ vertices and $(H,+)$ be an abelian group. What is the minimum size ${\sf S}_H(G)$ of the set of all sums $A(u)+A(v)$ over all injections $A:V(G)\to H$? In 2012, the first author, Angel, the second author, and Lubetzky proved that, for expander graphs and $H=\mathbb{Z}$, this minimum is at least $\Omega(\log n)$, and this bound is tight --- there exists a regular expander $G$ with  ${\sf S}_{\mathbb{Z}}(G)=O(\log n)$. We prove that, for every constant $d\geq 3$, the random $d$-regular graph $\mathcal{G}_{n,d}$ has significantly larger sum-sets: with high probability, for every abelian group $H$, ${\sf S}_H(\mathcal{G}_{n,d})=\Omega(n^{1-2/d})$. In particular, this proves that, for every $\varepsilon>0$, there exists a regular graph with $O(n)$ edges and with sum-sets of size at least $n^{1-\varepsilon}$, for all abelian groups. 

The bound ${\sf S}_H(\mathcal{G}_{n,d})=\Omega(n^{1-2/d})$ is tight up to a polylogarithmic factor: We show that, for every $3\leq d\leq \ln n/ \ln \ln n$, there exists an abelian group $H$ such that, for every graph $G$ on $n$ vertices with maximum degree at most $d$, ${\sf S}_H(G) \leq n^{1-2/d}(\log n)^{O(1)}$.

We also prove that, for $d\gg\ln^2 n$, with high probability, for every abelian group $H$, ${\sf S}_H(\mathcal{G}_{n,d})=n(1-o(1))$ and determine the second-order term, up to a polylogarithmic factor.
\end{abstract}

\section{Introduction}

For $d\geq 3$ and a positive integer $n$ such that $dn$ is even, we denote by $\mathcal{G}_{n, d}$ a uniformly random $d$-regular graph on the set of vertices $[n] := \{1, \ldots, n\}$. Everywhere in the paper we assume that $d=d(n)$ is not necessarily a constant and we study asymptotic properties of $\mathcal{G}_{n, d}$ as $n\to\infty$.

For a graph $G := (V(G), E(G))$ and an arbitrary abelian group $(H,+)$ let $A:V(G) \rightarrow H$ be an injection. For an edge $e = \{u, v\}\in E(G)$, we let $A(e) := A(u) + A(v)$. The sum-set of $G$ over $H$ is the minimum value of $|\{A(e):\,e\in E(G)\}|$ over all injections $A$.

The first author, Angel, the second author, and Lubetzky~\cite{AABL} studied sum-sets of expander graphs. They proved that for any $\delta$-(edge)-expander $G$, the sum-set of $G$ over $\mathbb{Z}$ equals $\Omega(\log n)$ and this bound is best possible --- there exists a $d$-regular $\delta$-expander such that its sum-set over $\mathbb{Z}$ equals $O(\log n)$; the respective witness has degree $d=\Theta(\log n)$.

In this paper, we prove that, for typical regular graphs, the logarithmic lower bound is very far from the answer. From our results it follows that, for every $\varepsilon>0$, there exists a constant $d$ such that whp\footnote{With high probability, that is, with probability approaching 1 as $n\to\infty$.} the sum-set of $\mathcal{G}_{n,d}$ has cardinality bigger than $n^{1-\varepsilon}$ over {\it any abelian group} $H$. It is worth noting that constructing sparse graphs with large sum-sets is not a trivial task. Actually, we are not aware of any deterministic construction of a $d$-regular graph $G$ with constant $d$ and sum-set of size at least $n^{1-\varepsilon}$. An example of a graph with a polynomial in $n$ sum-set was given by the first author, Angel, the second author, and Lubetzky in the same paper~\cite{AABL}. They observed that for a graph $G$ with $\mu_k$ cycles of an odd constant length $k\geq 3$, the sum-set over $\mathbb{Z}$ has size $\Omega(\mu_k^{1/k})$, and this bound is asymptotically tight. In particular, a disjoint union of triangles has sum-set over $\mathbb{Z}$ of size $\Theta(n^{1/3})$.

Below, for the sake of simplicity of presentation, we state an abbreviated, qualitative version of our result establishing the asymptotic behaviour of the minimum sum-sets of $\mathcal{G}_{n,d}$ for all $d=d(n)\geq 3$. We postpone the full statement of the lower bounds to Section~\ref{sc:lower}. For convenience, we denote by ${\sf S}_H(G)$ the sum-set of $G$ over $H$ and let 
$$
{\sf S}(G):=\min_H{\sf S}_H(G).
$$
\begin{theorem}
    Let $dn$ be even, 
    $d=d(n)\geq 3$, $\mathbf{G}_n\sim\mathcal{G}_{n,d}$. Then whp
    \begin{itemize}
    \item $
    c n^{1-2/d}\leq{\sf S}(\mathbf{G}_n)\leq n^{1-2/d}\ln^4 n
    $ for any $3\leq d\leq \ln n/\ln\ln n$ and some universal constant $c>0$;
    \item ${\sf S}(\mathbf{G}_n)=n/(\ln n)^{O(1)}$ for $\ln n/\ln\ln n<d = O(\ln^ 2n)$;
    \item ${\sf S}(\mathbf{G}_n)=\Theta(n)$ for $d = \Omega(\ln^2 n)$;
    \item ${\sf S}(\mathbf{G}_n) = n(1-o(1))$ for $d= \omega(\ln^2 n)$ and, more precisely, $n(1-C\ln^2 n/d)\leq {\sf S}(\mathbf{G}_n) \leq n(1-c/d)$ for some universal constants $c,C>0$.
    \end{itemize}
    \label{th:1-short}
\end{theorem}

In particular, when $d=\omega(1)$, we get that whp ${\sf S}(\mathbf{G}_n)=n^{1-o(1)}$, in contrast to the worst-case scenario for expander graphs with $d=\Theta(\log n)$. Also, for $d=3$, our lower bound matches the largest known cardinality of a sum-set of a sparse regular graph, achieved by a disjoint union of triangles. 

The upper bounds in Theorem~\ref{th:1-short} follow from our second main result that asserts an upper bound on ${\sf S}(G)$ for an {\it arbitrary} (not necessarily regular) graph $G$ on $n$ vertices with maximum degree $d$. Actually, this result shows that typical $d$-regular graphs $G$ maximise ${\sf S}(G)$, up to a polylogarithmic factor, in the class of graphs with maximum degree $d$.

\begin{theorem}
\label{th:upper_deterministic}
    Let $d=d(n)\geq 3$ and let $G=G(n)$ be a graph on $[n]$ with maximum degree $d$.
\begin{itemize}
 \item There exists a constant $C>0$ such that, for all $d\leq \ln n/\ln\ln n$ and all large enough $n$, ${\sf S}(G)\leq C n^{1-2/d} d^2 (\ln n)^{2+4/d}$.
 \item For all $d$ and all $n$, ${\sf S}(G)\leq n-\lceil n/(2d)\rceil+1$.
\end{itemize}
\end{theorem}

Let us observe that Theorem~\ref{th:upper_deterministic} indeed implies the upper bounds in Theorem~\ref{th:1-short}. To see this, note that we only need to prove the following for $\mathbf{G}_n\sim\mathcal{G}_{n,d}$:
\begin{itemize}
 \item Whp ${\sf S}(\mathbf{G}_n)\leq n^{1-2/d}\ln^4 n$ when $d\leq\ln n/\ln\ln n$.
 \item Whp ${\sf S}(\mathbf{G}_n)\leq n(1-1/(3d))$ when $d\gg \ln^2 n$.
\end{itemize}
The first bound follows immediately from the first part of Theorem~\ref{th:upper_deterministic}, and the second bound follows from the second part of Theorem~\ref{th:upper_deterministic}.


\paragraph{Proof strategy.} 
  The main complication in deriving lower bounds in Theorem~\ref{th:1-short} is that they require to overcome the union bound over the infinite set of groups and injections. We resolve this challenge by showing that, for a graph with diameter $D$, the problem reduces to quotient groups $\mathbb{Z}^k/F$, where $F\subset\mathbb{Z}^k$ is spanned by vectors with $L_1$-norms that do not exceed $3D$. Then, using the theory of lattices, we get that our reduction leads to a union bound over a fairly small set of abelian groups. Indeed, a set of generators for the lattice $F$ can be constructed recursively, one generator at a time. Since, at every step, we get a lattice with an integer volume, and every step reduces the volume by at least half, we immediately get an upper bound on the number of steps, which, in turn, implies an upper bound on the number of lattices $F$, which is sufficient to our goals.

In order to prove the first part of Theorem~\ref{th:upper_deterministic}, we construct a Cayley sum-graph of degree $O\left(n^{1-2/d}d^2(\ln n)^{2+4/d}\right)$ which is universal for the family of all graphs on $n$ vertices with maximum degree at most $d$. Clearly, the existence of such a graph immediately implies the desired assertion. In~\cite{AC}, the first author and Capalbo constructed a universal graph for this family with $n$ vertices and at most $O_d\left(n^{2-2/d}(\ln n)^{4/d}\right)$ edges. The main challenge here is to show that the construction from~\cite{AC} can be modified to get a Cayley sum-graph with the same universality property --- see a detailed overview in the beginning of Section~\ref{sc:upper}. The second part of Theorem~\ref{th:upper_deterministic} follows from the fact that, for any abelian group of order $n$ and its Cayley sum-graph $G'$ on $[n]$ with a generating set of size less than $n/(2d)$, $G$ can be drawn on $[n]$ edge-disjointly form $G'$. The latter assertion is an almost immediate corollary of~\cite[Theorem 4.2]{Catlin} and~\cite[Theorem 3]{SauerSpencer}.

\paragraph{Related work.} The study of sums and products along the edges of graphs dates back to the celebrated paper of Erd\H{o}s and Szemer\'{e}di~\cite{ES}, where they proposed a conjecture concerning the minimum possible value of $\max\{{\sf S}_{\mathbb{Z}}(G),{\sf P}_{\mathbb{Z}}(G)\}$, where ${\sf P}_R(G)$ denotes the product-set of $G$ over a given ring $R$, defined analogously to the sum-set. This conjecture --- particularly its special case $G=K_n$ --- has attracted significant attention, see, e.g., the comprehensive survey~\cite{GS-survey}. The general version of the conjecture for $n$-vertex graphs with at least $n^{1+\varepsilon}$ edges was refuted by the first author, Ruzsa, and Solymosi~\cite{ARS}. 
Since the foundational paper of Erd\H{o}s and Szemer\'{e}di, different modifications of the sum-product problem have been extensively studied, leading to various applications across diverse areas of mathematics.

\paragraph{Organisation.} In Section~\ref{sc:pre}, we recall some basic results from the theory of random regular graphs and the theory of lattices, that we use in our proofs. In the beginning of Section~\ref{sc:lower}, we state an unabbreviated version of the lower bounds from Theorem~\ref{th:1-short}, which gives a general lower bound for all $d=d(n)\geq 3$. We prove this stronger result in the same section. Section~\ref{sc:upper} contains a proof of Theorem~\ref{th:upper_deterministic}. In Section~\ref{sec:discussions} we discuss possible improvements of our results and remaining challenges.


\paragraph{Notation.} Everywhere in the paper we denote by $\log n$ and $\ln n$ the binary logarithm and the natural logarithm of $n$, respectively. For a positive integer $n$, we denote $[n]:=\{1,\ldots,n\}$. 
 For $f= (x_1, \ldots, x_k) \in \mathbb{Z}^k$, let $|f|_1$ denote the $L_1$-norm of $f$, i.e., $|f|_1 = \sum_{i=1}^k |x_i|$.

\begin{remark}
The first version of this paper 
 uploaded to the arXiv on 1.07.2025 
 did not include Theorem~\ref{th:upper_deterministic} and contained a much weaker linear in $n$ upper bound on ${\sf S}(\mathcal{G}_{n,d})$ that was proved using the second moment approach. 
\end{remark} 


\section{Preliminaries}
\label{sc:pre}

\subsection{Random regular graphs}

We will need asymptotic bounds on the total number of $d$-regular graphs and the number of $d$-regular graphs that avoid edges of a given graph.

\begin{claim}[\cite{McKay-2,McKay}]
\label{cl:mckay}
For any integer $3\leq d=O(\sqrt{n})$,  the number of $d$-regular graphs on $[n]$ equals
$$
g_d(n)=\frac{(nd)!}{(nd/2)!2^{nd/2}(d!)^n}\cdot e^{O(d^2)}.
$$
Moreover, let $\mathbf{G}_n\sim\mathcal{G}_{n,d}$. Let $G$ be a graph on $[n]$ with maximum degree $\Delta=o(n)$. Then
$$
\mathbb{P}(G\cap\mathbf{G}_n=\varnothing)= \exp\left(-\frac{d|E(G)|}{n}+O\left(\frac{d\Delta^2}{n}\right)\right).
$$
\end{claim}

Let us recall that whp random regular graphs with constant $d\geq 3$ are connected, hamiltonian, and their diameters are bounded by $2\log  n/\log d$~\cite{BF82,RW} (see also Sections 7.6 and 10.3 in~\cite{B01}). This result extends to growing $d$, due to the expansion properties of $\mathcal{G}_{n,d}$~\cite{BFSU} and due to the recent coupling results --- see~\cite[Theorem 1.2]{HLMPW} for even $n$ and~\cite[Theorem 8]{Gao},~\cite[Theorem 1.2]{GIM} to extend this result to odd $n$ as well.

\begin{theorem}[\cite{BFSU,HLMPW,RW}]
\label{th:diameter}
    Let $3\leq d=O(\sqrt{n})$, $\mathbf{G}_n \sim \mathcal{G}_{n, d}$. Then whp $\mathbf{G}_n$ is hamiltonian and its diameter equals $O( \log n/ \log d)$. 
\end{theorem}

\subsection{Lattices}

For $V\subset\mathbb{Z}^k$ let $\mathrm{Span}(V)$ denote the linear span of vectors from $S$ over the field $\mathbb{Z}$. When vectors in $V$ are independent, we refer to $\mathrm{Span}(V)$ as {\it $V$-lattice} and to $V$ as a {\it basis} of the lattice. {\it Dimension} of a $V$-lattice is the number $|V|$ of elements in its basis. We recall that for any finite subset $V'\subset\mathbb{Z}^k$, there exist a set of independent vectors $V\subset\mathbb{Z}^k$ such that $\mathrm{Span}(V')$ is $V$-lattice. A {\it fundamental domain} $D(V)$ of a $V$-lattice is a minimum set such that its $V$-translations cover the entire $\mathbb{R}^k$, i.e.
$$
 D(V)=\left\{\sum_{v\in V}t_vv,\,t_v\in[0,1)\right\}.
$$
The {\it volume} $\mathrm{Vol}(L)$ of a lattice $L$ is the volume of its fundamental domain. Let us note that $V$-lattice can also be $U$-lattice for some other basis $U$. Nethertheless, it is known that any two bases of the same lattice are related by a unimodular matrix. In other words, let $M_V$ be the matrix with rows being the basis vectors of $V$ and let $M_U$ be defined analogously, then $U$ is a basis of the lattice if and only if $V = SU$ for an integer $|V|\times |V|$-matrix $S$ with determinant $\pm 1$ (see~\cite[Lemma~16.1.6]{G12}). Therefore, the volume is independent of the choice of the basis.

A lattice $\tilde L$ is a {\it sublattice} of a lattice $L$, if $\tilde L\subset L$. If $\tilde L\subset L$ are lattices of the same dimension, the ratio $(L:\tilde L):=\frac{\mathrm{Vol}(L)}{\mathrm{Vol}(\tilde L)}$ is always an integer, which is called the index of $\tilde L$. If inclusion $\tilde L\subset L$ is proper, then $(\tilde L:L)>1$. Indeed, let $V$ and $\tilde V$ be arbitrary bases of $L$ and $\tilde L$ correspondingly and let $M_V$ and $M_{\tilde V}$ be defined as above. Since $\tilde L$ is a sublattice, every vector from $V$ is a linear combination over $\mathbb{Z}$ of vectors from $\tilde V$. In other words, there is an integer (not necessarily unimodular) $|V|\times |V|$-matrix $S$ such that $M_V = SM_{\tilde V}$. So, the volume of $D(V)$ differs by the factor of $|\det S|$ from the volume of $D(\tilde V)$. Finally, $|\det S| \neq 1$ since, by \cite[Lemma~16.1.6]{G12}, $\det S = \pm1$ implies that $V$ is a basis of $\tilde L$ and, hence, $L = \tilde L$. For more details about lattices, see, e.g.,~\cite{L09,G12}.

\section{Lower bound}
\label{sc:lower}

In this section we prove the following unabbreviated version of the lower bounds from Theorem~\ref{th:1-short}. 

\begin{theorem}
    \label{th:lower_main}
    Let $dn$ be even, 
    $d=d(n)\geq 3$, $\mathbf{G}_n\sim\mathcal{G}_{n,d}$.
    \begin{itemize}
        \item There exist a positive sequence  $\delta_n=o(1)$ and a constant $c>0$ such that the following bound holds: if $d\leq (2-\delta_n)\ln n/\ln\ln n$, then ${\sf S}(\mathbf{G}_n)\geq c n^{1-2/d}$ whp.
        \item For every $\delta>0$, there exists $C>0$, such that the following bound holds:  if $d>(2+\delta)\ln n /\ln\ln n$, then  ${\sf S}(\mathbf{G}_n) \geq n/w$ whp, where $w$ is the unique solution of the equation $w\ln w = C \ln^2n/d$.
    \end{itemize} 
\end{theorem}

\begin{proof}[Proof of the lower bounds in Theorem~\ref{th:1-short}.]

The lower bounds for $d\leq (2-\delta_n)\ln n/\ln\ln n$ follow immediately. We then observe that the bound ${\sf S}(\mathbf{G}_n)\geq n/(\ln n)^{1+o(1)}$ actually holds whp {\bf for all} $d=(2\pm o(1))\ln n/\ln\ln n$ --- this follows from Theorem~\ref{th:lower_main} and the coupling 
$(\mathbf{G}_n^1\sim\mathcal{G}_{n,d_1},\mathbf{G}_n^2\sim\mathcal{G}_{n,d_2})$, where $d_1=\Theta(\ln n/\ln\ln n)$, $d_1<d_2<d_1(1+o(1))$, and $\mathbf{G}_n^1\subset\mathbf{G}_n^2$ whp~\cite[Theorem 7]{Gao}. This gives the required bounds:
\begin{itemize}
    \item ${\sf S}(\mathbf{G}_n)=\Omega(n/\ln^2 n)$ whp for all $d\geq \ln n/\ln\ln n$, 
    \item ${\sf S}(\mathbf{G}_n)=\Theta(n)$ whp for all $d=\Omega(\ln^2 n)$, and 
    \item ${\sf S}(\mathbf{G}_n)=n(1-O(\ln^2n/d))$ whp for all $d\gg \ln^2 n$.
\end{itemize}
\end{proof}

The rest of the section is devoted to the proof of Theorem~\ref{th:lower_main}. One of the main challenges here is to reduce the set of groups in order to apply the union bound over $H$. We actually do this by imposing some requirements on $H$ that do not affect the union of events. Namely, in Claim~\ref{cl:lower_dec} below we show that any abelian group $H$ and an injection $A:[n] \rightarrow H$ can be replaced by some canonical $\tilde{H}$ and $\tilde{A}: [n] \rightarrow \tilde{H}$ such that 
$$
|\{\tilde A(e):\,e\in E(\mathbf{G}_n)\}|\leq|\{A(e):\,e\in E(\mathbf{G}_n)\}|.
$$
Then, in Claim~\ref{cl:lower_cl} we show an upper bound on the number of canonical groups $\tilde{H}$. We will use these claims to conclude the proof of Theorem~\ref{th:lower_main} in Section~\ref{sc:lower-proof}.

\subsection{Reduction}

Let us recall that any abelian group $(H,+)$ can be trivially considered as a $\mathbb{Z}$-module: 
\begin{itemize}
\item for $h \in H$, let $0_{\mathbb{Z}} \cdot h:=0_H$, 
\item for an integer $z > 0$, let $z \cdot h:=h + \ldots + h$, and 
\item $(-z) \cdot h := - (z \cdot h)$. 
\end{itemize}
Thus, for any $k>0$, any $f = (z_1, \ldots, z_k) \in \mathbb{Z}^k$ can be viewed as a $\mathbb{Z}$-module homomorphism $f:H^k \rightarrow H$, i.e. for $(a_1, \ldots, a_k) \in H^k$, we define $f \cdot (a_1, \ldots, a_k) := \sum_{i=1}^k z_i \cdot a_i$.

Let us note, that for any injection $A:[n]\rightarrow H$, we can consider instead an injection $\tilde{A}(x) := A(x) - A(1)$. Indeed, the sum-sets for both injections would have exactly the same cardinality, for any graph $G$ on $[n]$, since, for any $e_1,e_2\in E(G)$, $A(e_1)\neq A(e_2)$ if and only if $\tilde A(e_1)\neq\tilde A(e_2)$. Therefore, in what follows, without loss of generality, we assume that $A(1) = 0_H$.

\begin{claim}  
\label{cl:lower_dec}
    Let $D\geq 1$ and let $G$ be a connected graph on $[n]$ with diameter less than~$D$. Then, for any abelian group $(H,+)$ and any injection $A:[n]\rightarrow H$ such that $A(1) = 0_H$, the sum-set $A(E(G))=:\{a_1, \ldots, a_k\}$ satisfies the following. Let $\{e_1, \ldots, e_k\}$ be the standard basis in $\mathbb{Z}^k$. Then there exist  
    \begin{itemize}
        \item[$C$1] vectors $f_1, \ldots, f_m \in \mathbb{Z}^k$ with $|f_t|_1 \leq 3D$, for every $t\in[m]$, where $m=|E(G)| - (n-1)$, and
        \item[$C$2] an injection $\tilde A: [n] \rightarrow \tilde H$, where $\tilde H = \mathbb{Z}^k/\mathrm{Span}(\{f_1, \ldots, f_m\})$,
    \end{itemize}
    such that the quotient map $\pi: \mathbb{Z}^k \rightarrow \tilde H$ satisfies $\tilde A (E(G)) \subseteq \pi(\{e_1, \ldots, e_k\})$. 
\end{claim}

\begin{proof}

 Let $\varphi:\{e_1, \ldots, e_k \} \rightarrow \{a_1, \ldots, a_k \}$ denote the `identity' bijection, defined by $\varphi(e_i) = a_i$ for all $i \in [n]$. Let us fix an arbitrary spanning tree $T$ of $G$ such that the distance from the vertex $1$ to all the other vertices is less than $D$. Hence, between any two vertices in $T$ there is a unique path of length less than $2D$. 

Let us define an auxiliary function $A':[n] \rightarrow \mathbb{Z}^k$ recursively: First, we let $A'(1) = 0_{\mathbb{Z}^k}$ and then, at every step $\ell$, we define $A'(v)$ for all vertices $v$ that are at distance exactly $\ell$ in $T$ from 1. Assume that $\ell\geq 1$ and that $A'(v)$ is defined for all $v$ such that $d_T(v,1)\leq\ell-1$. Now, fix a vertex $v$ such that $d_T(v,1)=\ell$ and let $u$ be the neighbour of $v$ that belongs to the path that joins $v$ and 1 in $T$. We then define $A'(v)$ in the following way:
\begin{equation}
    A'(v) := \varphi^{-1} \Bigl(A\bigl(\{u,v\}\bigr)\Bigr) - A'(u). 
    \label{eq:A'}
\end{equation}

Let us recall that our goal is to define $\tilde H$ and an injection $\tilde A:[n]\to\tilde H$ so that $|\tilde A(E(G))|\leq|A(E(G))|$. In fact, the function $A'$ that we have just defined is an injection and it satisfies $|A'(E(T))|\leq|A(E(T))|$. Nevertheless, a similar inequality may not hold for the entire set of edges of the graph $G$. In order to overcome this issue, we will define $\tilde H$ as a quotient subspace of $\mathbb{Z}_k$ (satisfying C1 and C2), and then set $\tilde A := \pi \circ A'$, where $\pi$ is the respective quotient map.

As required in the claim, we will first define vectors $f_1,\ldots,f_m$. Then we will set~$\tilde H:=\mathbb{Z}^k/\mathrm{Span}(\{f_1, \ldots, f_m\})$ and, for the quotient map $\pi:\mathbb{Z}^k\to\tilde H$, we will define $\tilde A=\pi\circ A'$.  
 Let $g_1, \ldots, g_m$ be all the edges from $E(G) \setminus E(T)$. An ideal scenario for us is when, for every edge $g_t \in E(G) \setminus E(T)$, we have that $A'(g_t) = \varphi^{-1}(A(g_t))$. Indeed, in this case, $A'(E(G))=\varphi^{-1}(A(E(G)))$. 
  However, this might not hold, and so we set 
\begin{equation}
f_t := A'(g_t)-\varphi^{-1}(A(g_t)),
\label{eq:ft}
\end{equation}
and then
\begin{align*}
    \tilde A(g_t) =\pi(A'(g_t)) 
    &= \pi(f_t) + \pi(\varphi^{-1}(A(g_t))) \\
    &= \pi(\varphi^{-1}(A(g_t))) \in \pi(\varphi^{-1}(\{a_1, \ldots, a_k\})) = \pi(\{e_1, \ldots, e_k\}),
\end{align*}
as needed.

Let us stress that the choice of vectors $f_t$ has a natural interpretation in terms of the fundamental group $\pi_1(G)$ of the 1-dimensional clique complex of $G$. Indeed, for every loop $\pi_1(G,1)$ based at 1, 
 the alternating sum of $A(\{u,v\})$ over the edges $\{u,v\}$ of the loop, equals 0. As our aim is to bound the cardinality of $A(E(G))$, the following question arises: for an arbitrary  mapping $A^{e}:E(G) \rightarrow H$, does there exist a (not necessarily injective) mapping $A^v:E(G) \rightarrow H$ such that $A^v(1)=0$ and $A^v(\{x,y\}) = A^e(\{x,y\})$ for all edges $\{x,y\}\in E(G)$? It is well known that the answer is positive if and only if, for any loop  from $\pi_1(G,1)$, its alternating sum is zero. Since $\pi_1(G)$ is isomorphic to the free group generated by the edges $g_1, \ldots, g_m$, the requirements $\tilde A(f_1) = \ldots = \tilde A(f_m) = 0$ that define the kernel are, in a sense, weakest possible.

We have already shown that $\tilde A(g) \in \pi(\{e_1, \ldots, e_k\})$, for every edge $g\in E(G) \setminus E(T)$. The same statement trivially holds for every edge $g\in E(T)$:
\begin{align*}
    \tilde A(g) =\pi(A'(g)) \stackrel{\eqref{eq:A'}}= \pi(\varphi^{-1}(A(g)))  \in \pi(\varphi^{-1}(\{a_1, \ldots, a_k\})) = \pi(\{e_1, \ldots, e_k\}).
\end{align*}
So, it remains to prove the following assertions from C1 and C2:
\begin{itemize}
\item $|f_t|_1\leq 3D$, for every $t\in[m]$;
\item  $\tilde A$ is an injection.
\end{itemize}

We start with the first assertion. By the definition~\eqref{eq:A'}, for any edge $\{u,v\}\in E(T)$, there exists $i \in [k]$ such that $A'(u) + A'(v) = e_i$. So, $|A'(u)|_1$ and $|A'(v)|_1$ differ by one. By induction on the distance from the vertex 1 of the tree $T$, we get that every vertex $u\in[n]$ has $$
|A'(u)|\leq d_T(1,u)+|A'(1)|_1 <D,
$$
as $|A'(1)|_1 = 0$. So, by the definition~\eqref{eq:ft}, for any $t \in [m]$, the respective edge $g_t=\{u,v\}$ of $G$, and large enough $n$,
$$
|f_t|_1\leq|A'(g_t)|_1+|\varphi^{-1}(A(g_t))|_1 = |A'(\{u,v\})|_1+1 \leq |A'(u)|_1+|A'(v)|_1+1< 2D +1 \leq 3D,
$$
as required.

In order to prove that $\tilde A$ is an injection, let us show the following auxiliary statements:
\begin{itemize}
    \item[S1]$A'(i) \cdot (a_1, \ldots, a_k) = A(i)$ for every $i \in [n]$,
    \item[S2]$f_t \cdot (a_1, \ldots, a_k) = 0_H$ for every $t \in[m]$,
    \item[S3] $\pi^{-1}\left(\tilde{A}(i)\right) \cdot (a_1, \ldots, a_k)$ is well-defined and equals $A(i)$ for every $i \in [n]$.
\end{itemize}

Due to the definition of the dot product,
 for any $j \in [k]$, 
\begin{equation}
    \varphi^{-1} (a_j) \cdot (a_1, \ldots, a_k) = e_j \cdot (a_1, \ldots, a_k) = a_j.
    \label{eq:aux}
\end{equation}  

We prove S1 by induction on the distance $\ell$ between the considered vertex $i$ and the vertex 1 in the tree $T$. 
 For $\ell = 0$, 
 $$
 A'(1) \cdot (a_1, \ldots, a_k)
 = 0_{\mathbb{Z}^k}\cdot (a_1, \ldots, a_k)= 0_H = A(1).
 $$
 Now, suppose $\ell > 0$ and that S1 is proven  for all vertices at distance at most $\ell-1$ from 1. Let $i$ be a vertex with $d_T(1, i) = \ell$, and let $j$ be its neighbour in $T$ with $d_T(1, j)=\ell-1$. Then, S1 holds for $j$ and, hence,
 $$
 A'(i) \cdot (a_1, \ldots, a_k) \stackrel{\eqref{eq:A'}}= \biggl(\varphi^{-1} \Bigl(A\bigl(\{i,j\}\bigr)\Bigr) - A'(j)\biggr) \cdot (a_1, \ldots, a_k) \stackrel{\eqref{eq:aux}}= A(\{i,j\}) - A(j) = A(i),
 $$
completing the proof of S1.

S2 holds since, for any $t \in [m]$, by the definition~\eqref{eq:ft},
$$
f_t \cdot (a_1, \ldots, a_k) = A'(g_t)\cdot (a_1, \ldots, a_k) - \varphi^{-1} (A(g_t)) \cdot (a_1, \ldots, a_k) \stackrel{S1, \eqref{eq:aux}}= A(g_t) - A(g_t) = 0_H.
$$ 

Finally, let us prove S3. First, note that $\pi^{-1}\left(\tilde{A}(i)\right) \cdot (a_1, \ldots, a_k)$ is well-defined, since any two vectors from $\pi^{-1}\left(\tilde{A}(i)\right)$ differ by a vector $f\in \text{Span}(\{f_1, \ldots, f_m\})$ and, by S2, $f \cdot (a_1, \ldots, a_k) = 0_H$. Now, since $A'(i) \in \pi^{-1}\left(\tilde A (i)\right)$, we get 
$$
\pi^{-1}\left(\tilde{A}(i)\right) \cdot (a_1, \ldots, a_k) = A'(i) \cdot (a_1, \ldots, a_k) \stackrel{S1}= A(i),
$$
completing the proof of S3.

The fact that $\tilde A$ is an injection follows immediately from S3. Indeed, for any $u, v \in [n]$, $u \neq v$, the equality $\tilde A(u) = \tilde A(v)$ implies $A(u) = A(v)$ by S3, which is impossible since $A$ is an injection. This completes the proof of the claim.
\end{proof}

\begin{claim}
    There exists $C>0$ such that the following holds. Let $k,D,m\geq 2$ be integers.  Let 
    \begin{equation}
    \label{def:F-system}
    \mathcal{F} := \{\mathrm{Span}(f_1, \ldots, f_m) \mid f_t \in \mathbb{Z}^k,\ |f_t|_1 < 3D\text{, for }t \in [m]\}.
    \end{equation}
Then $|\mathcal{F}| \leq k^{C\cdot k D \log D}$.
\label{cl:lower_cl}
\end{claim}

\begin{proof}

Consider an arbitrary set $F := \{f_1, \ldots, f_m\}$ such that $ f_t \in \mathbb{Z}^k,\ |f_t|_1 <3D\text{, for }t \in [m]$. It is sufficient to prove that there exists a set $S\subset\mathbb{Z}^k$ such that 
\begin{itemize}
    \item $\text{Span}(S) = \text{Span}(F)$;
    \item $|s|_1 \leq 3D$, for every $s \in S$;
    \item $|S| \leq k(2+\log D)$.
\end{itemize}
Indeed, for a large enough constant $C_0$, there are in total at most 
$$
N := \sum_{s\leq\min\{k,3D\}}{k\choose s}{3D+s\choose s}2^s\leq
\sum_{s\leq 3D}\frac{k^s}{s!}\cdot 2^{3D+2s}\leq(C_0 k)^{3D}
$$ 
vectors in $\mathbb{Z}^k$ with $L_1$-norm not exceeding $3D$. So, $|\mathcal{F}|$ does not exceed the number of different sets $S$ described above which, in turn, does not exceed 
$$
 1 + N + \ldots + N^{\lfloor  k(2+\log D) \rfloor}<  N^{\lfloor  k(2+\log D) \rfloor+1} \leq (C_0k)^{C_1 Dk\log D}\leq k^{CkD\log D},
$$
for some constants $C>C_1>0$.

For a given set $ F $, we construct the set $ S $ as follows. Start by selecting $r := \text{rank}(\text{Span}(F))$ linearly independent vectors from $F$ and include them in $ S $. Denote this initial set as $S_{r}$. We then construct sets  $ S_r, S_{r+1}, \ldots $ one by one. Suppose we have already constructed a set $ S_t $, $t\geq r$. In order to construct $S_{t+1}$, arbitrarily select an element from $ F \setminus S_t $ and add it to $ S_t $ such that the span of the new set, $ \mathrm{Span}(S_{t+1}) $, is strictly larger than $ \mathrm{Span}(S_t) $. If no such element can be chosen, terminate the process and set $ S = S_t $. Since $S\subset F$ and $\mathrm{Span}(S) = \mathrm{Span}(F)$, it remains to prove that $t \leq k(2+\log D)$.

Consider the lattice $L_r$ spanned by $ S_r $. Observe that $ \mathrm{Vol}(L_r) \leq (3 D)^r $. It remains to recall that the volume of any lattice is an integer and that adding any element that changes the lattice reduces the volume by at least half. Then the process has to terminate in at most $$
\log((3D)^r) \leq k(\log 3+\log D) < k(2+\log D)
$$ 
steps, completing the proof.

\end{proof}

\subsection{Proof of Theorem~\ref{th:lower_main}}
\label{sc:lower-proof}

We prove Theorem~\ref{th:lower_main} separately for $d\leq\sqrt{n}$ and $d>\sqrt{n}$. The reason for such distinguishing is that the asymptotic estimations from Claim~\ref{cl:mckay} are valid only when $d=O(\sqrt{n})$. Nevertheless, for large $d$, we may rely on sandwiching results from~\cite{GIM,GIM-weak} and reduce the problem to binomial random graphs that are simpler to analyse. Therefore, the argument in the case $d>\sqrt{n}$ is conceptually the same, while technical details are significantly easier. For the sake of simplicity of exposition, we start from assuming that $\mathbf{G}_n$ is a random graph on $[n]$ and do not specify its {\it isomorphism-invariant} probability distribution $\mathcal{P}_n$. We only assume that $\mathcal{P}_n$ satisfies the following requirements:
\begin{itemize}
 \item whp $\mathbf{G}_n$ is Hamiltonian;
 \item whp $\mathbf{G}_n$ is connected and has diameter less than $d_n$, where $d_n$ is some specified sequence of positive integers.
\end{itemize}

Let $\mathcal{C}_n$ be the set of all $n$-cycles on $[n]$. For $C \in \mathcal{C}_n$, let $\mathbf{G}_n^C$  be distributed according to $\mathcal{P}_n$ conditioned on $C \subseteq \mathbf{G}_n^C$. Below we show, that, in order to prove Theorem~\ref{th:lower_main}, it is sufficient to establish strong enough upper bounds on the probability that Theorem~\ref{th:lower_main} does not hold for the random graph $\mathbf{G}_n^C$, for any cycle $C \in \mathcal{C}_n$. 

Let $A_n$ be an arbitrary isomorphism-invariant set of graphs on $[n]$. Denote by $D_n$ the set of all connected graphs on $[n]$ with diameter at most $d_n$. Let the cycle $C_0$ consist of all edges $\{1+i,1+(i\oplus 1)\}$, $i\in\{0,\ldots,n-1\}$, where $\oplus$ is the summation modulo~$n$. Then
\begin{align*}
\mathbb{P}(\mathbf{G}_n\in A_n)&
\leq\mathbb{P}(\mathbf{G}_n\in A_n, \,\mathbf{G}_n\in D_n, \,\mathbf{G}_n\text{ is Hamiltonian}) + o(1) \\
&\leq \sum_{C\in\mathcal{C}_n}\mathbb{P}(\mathbf{G}_n\in A_n, \,\mathbf{G}_n\in D_n, \, C \subseteq\mathbf{G}_n) +o(1)\\
&\leq \sum_{C\in\mathcal{C}_n}\mathbb{P}(\mathbf{G}_n^C\in A_n, \mathbf{G}^C_n\in D_n) \cdot \mathbb{P}(C \subseteq\mathbf{G}_n) + o(1)\\
&= \mathbb{P}(\mathbf{G}_n^{C_0}\in A_n, \, \mathbf{G}^{C_0}_n\in D_n)\cdot \sum_{C\in\mathcal{C}_n} \mathbb{P}(C \subseteq\mathbf{G}_n) +o(1)\\
&= \mathbb{P}(\mathbf{G}_n^{C_0}\in A_n, \, \mathbf{G}^{C_0}_n\in D_n)\cdot\mathbb{P}(C_0\subseteq\mathbf{G}_n)\cdot \frac{n!}{2n} + o(1).
\end{align*}
In what follows, we omit the subscript $0$ and write $C=C_0$. We now choose a positive integer $k^*=k^*(n)$ that will be specified later and set $A_n=\{G:{\sf S}(G)< k^*\}$. In order to complete the proof of the theorem, we will show that, for specified $k^*$, $d_n$, and $\mathcal{P}_n$, 
\begin{equation}
\mathbb{P}(\mathbf{G}_n^C\in A_n, \, \mathbf{G}^C_n\in D_n) = o\left((n^{n}\cdot \mathbb{P}(C\subseteq\mathbf{G}_n))^{-1}\right).
\label{eq:Th_lower_reduction-AD}
\end{equation}

For $k\leq k^*$, let $\mathcal{B}_k$ be the event that there exist an abelian group $H$ and an injection $ A: [n] \to H $ such that $ |A(E(\mathbf{G}_n^C))| = k $.  Due to Claim~\ref{cl:lower_dec}, the event $\cup_{k\leq k^*} \mathcal{B}_k \cap \{\mathbf{G}_n^C\in D_n\}$ implies $\cup_{k\leq k^*} \mathcal{B}'_k$, where $\mathcal{B}'_k$ is the event that there exist an abelian group $\tilde H$ and an injection $ \tilde A: [n] \to \tilde H $ satisfying the properties from the conclusion of Claim~\ref{cl:lower_dec} with $D:=d_n$. In particular, for the set $\mathcal{F}$ defined in~\eqref{def:F-system}, there exists $F \in \mathcal{F} := \mathcal{F}(k)$ such that $\tilde H \cong \mathbb{Z}^k / \text{Span}(F)$. We also note that the event $\mathcal{B}_k$ is monotone in $k$, that is $\cup_{k\leq k^*} \mathcal{B}'_k=\mathcal{B}'_{k^*}$. Therefore, in what follows we set $k=k^*$.

Let $F \in \mathcal{F}(k)$. Set $\tilde H = \mathbb{Z}^k / \text{Span}(F)$.
 Let $\pi : \mathbb{Z}^k \rightarrow \tilde H$ be the natural projection and $\{e_1, \ldots, e_k\}$ be the standard basis of $\mathbb{Z}^k$. 
Let
$$
\mathcal{A} =\mathcal{A}(\tilde H):= \{\tilde A:[n]\hookrightarrow \tilde H \mid  \tilde A(E(C)) \subseteq\{\pi(e_1), \ldots, \pi(e_k)\},\ \tilde A(1)=0_{\tilde H} \}.
$$
 Note that $|\mathcal{A}| \leq k^{n-1}$, since we can determine the value of $\tilde A$ on the vertices of the cycle one by one, starting from $\tilde A(1)=0_{\tilde H}$: at the $i$-th step, once we know $\tilde A(i)$, then $\tilde A(i+1)$ can only be equal to one of the following $k$ values: $\pi(e_1) - \tilde A(i), \ldots, \pi(e_k) - \tilde A(i)$. Therefore, due to Claim~\ref{cl:lower_cl}, there exists a constant $c>0$ such that it suffices to prove that 
 for an arbitrary tuple $(F\in\mathcal{F}(k),\tilde A\in\mathcal{A}(\tilde H))$,
\begin{equation}
 \mathbb{P}\biggl(\tilde A\left(E(\mathbf{G}_n^C)\right) \subseteq \{\pi(e_1), \ldots, \pi(e_k)\}\biggr)=o\left(k^{-c\cdot k \cdot D \cdot \log D-n}\cdot (n^{n}\cdot \mathbb{P}(C\subseteq\mathbf{G}_n))^{-1} \right),
\label{eq:th3.3-final-prob-bound}
\end{equation}
since this bound implies~\eqref{eq:Th_lower_reduction-AD}.

We now consider the case $d>\sqrt{n}$ and recall that, due to~\cite[Theorem 1.2]{GIM},~\cite[Theorem 1.5(c)]{GIM-weak}, there exists $p=(1-o(1))\frac{d}{n}$ and a coupling $(\mathbf{G}_n^1,\mathbf{G}_n^2)$ where $\mathbf{G}_n^1\sim\mathcal{G}(n,p)$ and $\mathbf{G}_n^2\sim\mathcal{G}_{n,d}$ such that $\mathbf{G}_n^1\subseteq\mathbf{G}_n^2$ whp. In order to complete the proof of Theorem~\ref{th:lower_main} in this case, it is sufficient to show that~\eqref{eq:th3.3-final-prob-bound} holds for $\mathbf{G}_n\sim\mathcal{G}(n,p)$ for every $p=(1-o(1))\frac{d}{n}$, since, for such $p$, $\mathbf{G}_n$ is Hamiltonian whp~\cite{B01}. We first observe that $\mathbb{P}(C\subseteq\mathbf{G}_n)=p^n\leq(d/n)^n$. Moreover, since whp $\mathbf{G}_n$ has diameter at most 3~\cite{B01}, we can set $D=4$. We also set $k:=\left\lfloor n(1-\ln^2 n/d)\right\rfloor$, with a room for improvement. Then, it remains to show that, for any $c>0$ and an arbitrary tuple $(F\in\mathcal{F}(k),\tilde A\in\mathcal{A}(\tilde H))$,
\begin{equation}
 \mathbb{P}\biggl(\tilde A\left(E(\mathbf{G}_n^C)\right) \subseteq \{\pi(e_1), \ldots, \pi(e_k)\}\biggr)=o\left(n^{-cn}\right).
\label{eq:th3.3-final-prob-bound-binomial}
\end{equation}
For every $i\in[n]$, the number of vertices $j\in[n]$ such that $\tilde A(\{i,j\})\notin\{\pi(e_1),\ldots,\pi(e_k)\}$ is at least $n-k\geq n\ln^2 n/d$. So, there are at least $(n\ln n)^2/(2d)$ edges $\{i,j\}$ such that $\tilde A(\{i,j\})\notin\{\pi(e_1),\ldots,\pi(e_k)\}$. We get
$$
 \mathbb{P}\biggl(\tilde A\left(E(\mathbf{G}_n^C)\right) \subseteq \{\pi(e_1), \ldots, \pi(e_k)\}\biggr)\leq (1-p)^{(n\ln n)^2/(2d)}
 \leq\exp\left(-(1-o(1))\frac{n\ln^2 n}{2}\right)=o(n^{-cn}),
$$
as required.

We now switch to the case $d\leq \sqrt{n}$ and set $\mathcal{P}_n:=\mathcal{G}_{n,d}$. We have  to show that~\eqref{eq:th3.3-final-prob-bound} holds for $\mathbf{G}_n\sim\mathcal{G}_{n,d}$. In this case, $D=d_n=\beta\frac{\ln n}{\ln d}$, where the constant $\beta>0$ comes from Theorem~\ref{th:diameter}. For $\mathbf{G}'_n\sim\mathcal{G}_{n,d-2}$, 
\begin{align*}
    \mathbb{P}(C \subseteq \mathbf{G}_n) &= \frac{g_{d-2}(n)\cdot\mathbb{P}(E(\mathbf{G}'_n)\cap E(C)=\varnothing)}{g_d(n)}\\
&\stackrel{\text{Claim}~\ref{cl:mckay}}\leq \frac{(nd)^{n}(d(d-1))^n}{(n(d-2))^{2n}}\cdot e^{O(d^2)}
    \stackrel{d \geq 3}\leq (6d)^n\exp(O(d^2))/n^n.
\end{align*}
Let us define $k$. Let $ \alpha$ be a large enough constant and $\delta_n=o(1)$ be a sequence  to be chosen later. Let a positive integer $k$ be the biggest integer satisfying the inequalities:
\begin{itemize}
    \item $ n > \alpha k^{d/(d-2)} $, if $d \leq (2-\delta_n)\ln n/\ln\ln n$;
    \item $n > wk$, if $(2+\delta)\ln n/\ln\ln n<d\leq\sqrt{n}$, where $w$ is the unique solution of the equation $w\ln w = \alpha \ln^2n/(d-2)$. 
\end{itemize}
Due to~\eqref{eq:th3.3-final-prob-bound}, it suffices to prove that, for every $c>0$, there exists $\alpha$ such that, for an arbitrary tuple $(F\in\mathcal{F}(k),\tilde A\in\mathcal{A}(\tilde H))$,
\begin{equation}
\label{eq:th3.3-final-prob-bound-regular}
 \mathbb{P}\biggl(\tilde A\left(E(\mathbf{G}_n^C)\right) \subseteq \{\pi(e_1), \ldots, \pi(e_k)\}\biggr)=o\left(k^{-c\cdot k \cdot D \cdot \log D-n}\cdot d^{-n}\cdot e^{-c(d^2+n)}\right).
\end{equation}
Let us show that the following two claims conclude the proof.

\begin{claim}
\label{cl:lower_1}
Let $\mathbf{G}'_n \sim \mathcal{G}_{n, d-2}$.
\begin{equation}
\mathbb{P}( E(\mathbf{G}'_n) \cap E(C)= \varnothing) = \exp(-(d-2)  + O(d/n)). \label{eq:3.1}
\end{equation}
\end{claim}

Claim~\ref{cl:lower_1} follows immediately from Claim~\ref{cl:mckay}.

\begin{claim}
\label{cl:lower_2}
Let $\alpha$ be a large positive constant. Let $\mathbf{G}'_n \sim \mathcal{G}_{n, d-2}$. Then,
\begin{enumerate}
\item for
all large enough $n$ and $d \leq (2-\delta_n)\ln n/\ln\ln n$, it holds that
\begin{equation}
\mathbb{P}\biggl(\tilde A(E(\mathbf{G}_n')) \subseteq \{\pi(e_1), \ldots, \pi(e_k)\} \biggr) \leq e^{-2 (c+1) dn}k^{-n};
\end{equation}
and 
\item for $d \geq (2+\delta)\ln n/\ln\ln n$, we have
$$
\mathbb{P}\biggl(\tilde A(E(\mathbf{G}_n')) \subseteq \{\pi(e_1), \ldots, \pi(e_k)\} \biggr) \leq e^{-c(d^2+n)}k^{-n-c k \log n} d^{-n}. 
$$
\end{enumerate}
\end{claim}
Indeed, 
\begin{align*}
  \mathbb{P}\biggl(\tilde A\left(E(\mathbf{G}_n^C)\right) \subseteq \{\pi(e_1), \ldots, \pi(e_k)\}\biggr)
  &=
    \mathbb{P}\biggl(\tilde A\left(E(\mathbf{G}'_n)\right) \subseteq \{\pi(e_1), \ldots, \pi(e_k)\}\mid
    E(\mathbf{G}'_n) \cap E(C)= \varnothing
    \biggr)\\
  &\leq\frac{
  \mathbb{P}\biggl(\tilde A\left(E(\mathbf{G}'_n)\right) \subseteq \{\pi(e_1), \ldots, \pi(e_k)\} \biggr)}{\mathbb{P}(
    E(\mathbf{G}'_n) \cap E(C)= \varnothing)
    }\\
&\stackrel{\text{Claim}~\ref{cl:lower_1}}<
  e^{(d-2) + o(1)}\cdot
  \mathbb{P}\biggl(\tilde A\left(E(\mathbf{G}'_n)\right) \subseteq \{\pi(e_1), \ldots, \pi(e_k)\} \biggr).
\end{align*}
 As $o(1) < 2$ for large $n$, it remains to show that 
$$
\mathbb{P}\biggl(\tilde A\left(E(\mathbf{G}'_n)\right) \subseteq \{\pi(e_1), \ldots, \pi(e_k)\} \biggr) = o\left(k^{-c\cdot k \cdot D \cdot \log D-n}\cdot d^{-n} \cdot e^{-c(d^2+n)}
\right).
$$
We further apply Claim~\ref{cl:lower_2} and consider separately the two cases from the claim. \\

{\bf 1.} In the case $d \leq (2-\delta_n)\ln n/\ln\ln n$, it suffices to show that
$$
e^{-dn} = o(k^{-c\cdot k \cdot D \cdot \log D}),
$$ 
since $e^{-(2c+1)dn}k^{-n}<k^{-n}d^{-n} e^{-c(d^2+n)}$, for $d \geq 3$. Then it suffices to show
$$
\frac{c\cdot  k \cdot  D \cdot \log D \cdot \ln k}{dn} \leq1/2.
$$
The bound follows immediately for $d \leq \frac{\ln n}{2\ln \ln n}$ and large $n$, since 
$$
\frac{k}{n}
\leq n^{-2/d}\leq \exp\left(-\frac{2\ln n }{\ln n / 2\ln \ln n}\right) = \exp(-4 \ln \ln n) =\ln^{-4}n,
$$
and $D = \beta\log n/ \log d$. For the remaining $d\in\left(\frac{\ln n}{2\ln\ln n},\frac{(2-\delta_n)\ln n}{\ln\ln n}\right]$, we have $D \leq O\left(\frac{\log k}{\log d}\right) \leq O\left(\frac{\log n}{\log d}\right)$, $k / n \leq \alpha^{-1} \exp(-2\ln n/ (d-2))$, and $\ln k / d \leq \ln n/d \leq 4\log\log n$. So,
  \begin{align*}
      \frac{c\cdot  k \cdot  D \cdot \log D \cdot \ln k}{dn} &=  O\left(\alpha^{-1} \exp(-2\ln n/ d) \cdot \log\log n \cdot \frac{\log n}{\log d} \cdot \log\left(\frac{\log n}{\log d}\right)\right) \\
      & = O\left(\alpha^{-1} \exp(-2\ln n/ d) \cdot \log\log n \cdot \log n\right) = O(\alpha^{-1}),
  \end{align*}
  where the last equality holds for $d \leq \frac{2\ln n}{(1+\varepsilon)\ln \ln n}$, for any $\varepsilon > 0$, and, in particular, it holds for $d \leq \frac{(2-\delta_n)\ln n}{\ln \ln n}$ for some specific $\delta_n=o(1)$ that approaches 0 sufficiently slow. So, taking $\alpha$ sufficiently large, we conclude the proof in the first case. \\

{\bf 2.} For the case $d > (2+\delta)\ln n/\ln \ln n$, note that $D\log D = O(\log n)$. The bound~\eqref{eq:th3.3-final-prob-bound} immediately follows. \\

It remains to prove Claim~\ref{cl:lower_2}.

\begin{proof}[Proof of Claim~\ref{cl:lower_2}.]

Let $G^{\tilde A}$ be the graph on
$[n]$ with every edge $e$ such that $\tilde A(e) \in \{\pi(e_1), \ldots, \pi(e_k)\}$. By definition, the degree of every vertex in $G^{\tilde A}$ does not exceed $k$. Also, $C$ is a subgraph of $G^{\tilde A}$ and the event $\{\tilde A(E(\mathbf{G}_n')) \subseteq \{\pi(e_1), \ldots, \pi(e_k)\}\}$ can be written as $\{\mathbf{G}_n'\subseteq G^{\tilde A}\}$.

    Due to Claim~\ref{cl:mckay}, the number of $(d-2)$-regular graphs on $[n]$ equals 
    $$
    g_d(n)\geq
    \exp\left(\frac{(d-2)n}{2}\ln \frac{n}{d-2}+\frac{n(d-2)}{2} -O(n\ln d)\right),
    $$
    for $n$ large enough.
    Moreover, since a graph on $[n]$ with maximum degree at most $k$ has at most $kn/2$ edges, the number of $(d-2)$-regular subgraphs of $G^{\tilde A}$ is at most 
\begin{align*}
{kn/2\choose (d-2)n/2} &\leq\exp\left(\frac{(d-2)n}{2}\ln \frac{k}{d-2}+\frac{(k-d+2)n}{2}\ln\left(1+\frac{d-2}{k-d+2}\right)\right)\\
&\leq
\exp\left(\frac{(d-2)n}{2}\ln \frac{k}{d-2}+\frac{n(d-2)}{2}\right),
\end{align*}
for $n$ large enough.
 Therefore,
\begin{align*}
\mathbb{P}(\mathbf{G}_n'\subseteq G^{\tilde A}&)\leq\exp\left(\frac{n(d-2)}{2}\ln\frac{k}{d-2} - \frac{n(d-2)}{2}\ln \frac{n}{d-2}+O(n\ln d)\right)\\
&=\exp\left(\frac{n(d-2)}{2}\ln\frac{k}{n} +O(n\ln d) \right).
\end{align*}

{\bf 1.} If $d \leq (2-\delta_n)\ln n/\ln \ln n$, then $k/n\leq(\alpha k^{2/(d-2)})^{-1}$. In this case, for sufficiently large $\alpha\geq\alpha(c)$,
\begin{align*}
    \mathbb{P}(\mathbf{G}_n'\subseteq G^{\tilde A})\leq \exp\left(-n\ln k-\frac{n(d-2)\ln\alpha}{2}+O(n\ln d)\right)\leq e^{-2(c+1)dn}k^{-n},
\end{align*}
as required.\\

{\bf 2.} Let $d > (2+\delta)\ln n/\ln\ln n$. Recall $\frac{n}{k}>w>\frac{n}{k+1}$, where $w\ln w=\alpha\ln^2n/(d-2)$. Thus,
$$
(d-2)\ln w=\frac{\alpha\ln^2 n}{w}=\Omega(\ln n),
$$
implying
\begin{align*}
    \mathbb{P}(\mathbf{G}_n'\subseteq G^{\tilde A})\leq \exp\left(-\frac{n(d-2)}{2}\ln w+O(n\ln d)\right), 
\end{align*}
where the first-order term satisfies the following:
$$
n(d-2)\ln w=\Omega(n\ln n)\gg n\ln d\quad\text{ and}
$$
$$
 n(d-2)\ln w\sim k(d-2)w\ln w=\alpha k\ln^2n.
$$
So, by choosing $\alpha$ sufficiently large, we achieve
$$\exp\left(-\frac{\delta}{6}\cdot \frac{n(d-2)}{2} \cdot \ln w\right) < e^{-c(d^2+n)} k^{-ck\log n} d^{-n}.$$
 It only remains to prove
$(1-\delta/3)\frac{n(d-2)}{2}\ln w> n\ln n\geq n\ln k$. Since $w$ decreases with $d$ and achieves its maximum $(1/(2+\delta)+o(1))\alpha \ln n$ when $d=(2+\delta)\ln n/\ln\ln n$. We get
$$
 \frac{n(d-2)\ln w}{2n\ln n}=
 \frac{\alpha \ln n}{2w}\geq
 \frac{1+o(1)}{2/(2+\delta)}>\left(1-\frac{\delta}{3}\right)^{-1},
$$
for sufficiently small $\delta$ and large $n$.

We finally get
\begin{align*}
    \mathbb{P}(\mathbf{G}_n'\subseteq G^{\tilde A})&\leq e^{-c(d^2+n)} k^{-ck\log n} d^{-n} \exp\left(-(1-\delta/6)\cdot\frac{n(d-2)\ln w}{2}+O(n\ln d)\right)\\
    &<e^{-c(d^2+n)} k^{-ck\log n} d^{-n} \exp\left(-\frac{1-\delta/6}{1-\delta/3}\cdot n\ln k+O(n\ln d)\right)\\
    &<e^{-c(d^2+n)} k^{-n-ck\log n} d^{-n},
\end{align*}
for sufficiently large $\alpha\geq\alpha(c,\delta)$, as required.
\end{proof}

This completes the proof of Theorem~\ref{th:lower_main}.

\section{Proof of Theorem~\ref{th:upper_deterministic}}
\label{sc:upper}





We prove the first part of Theorem~\ref{th:upper_deterministic} in Section~\ref{sc:proof_upper1}. The proof is based on a modification of a construction of a sparse universal graph for the family of all $n$-vertex graphs with maximum degree $d$ given in~\cite{AC}. The new modification transforms this construction to a Cayley sum-graph. Indeed, Theorem~\ref{th:upper_deterministic} part 1 would immediately follow from the fact that there exists a Cayley sum-graph of degree $O(n^{1-2/d}d^2(\ln n)^{2+4/d})$ which is universal for all graphs on $[n]$ with maximum degree at most $d$. This fact is stated as Theorem~\ref{t912} in Section~\ref{sc:proof_upper1}. Its proof largely follows the strategy from~\cite{AC}: start from a sparse expander $Z$ and then construct a graph $F=(Z^d,E)$, where two vertices are adjacent whenever at least two pairs of their corresponding coordinates are adjacent in $Z$. The latter condition ensures that, for every graph $G$ with maximum degree $d$, its (appropriately defined) random embedding into $F$ is nearly injective --- all pre-images have size $O(d\ln^2 n)$. Consequently, the blow-up graph $\Gamma$, obtained by replacing every edge of $F$ with a large enough balanced complete bipartite graph, must contain $G$ as a subgraph. The key novel ingredient in our proof is ensuring  that $\Gamma$ is a Cayley sum-graph of the required degree. Notably, each stage of the construction preserves the Cayley sum-graph property, so the only remaining task is to verify this for the initial expander $Z$. For this, we employ the Cayley expander construction from~\cite{AR}. An intriguing feature of the proof is that the
construction of the universal graph is explicit,
whereas the embedding of any $n$-vertex graph $G$ with maximum degree at most $d$ in it is probabilistic.

As for the second part of Theorem~\ref{th:upper_deterministic}, we show that, for every graph $G$ with maximum degree $d$, {\bf any} abelian group $H$ of order $n$ gives the required bound, i.e. ${\sf S}_H(G)\leq n-\lceil n/(2d)\rceil+1$. Otherwise, any set $U\subset H$ of size $\lceil n/(2d)\rceil-1$ has a representative in $\{A(e):e\in E(G)\}$ for every injection $A:[n]\to H$. This contradicts the fact that any two graphs on $[n]$ with small enough maximum degrees can be placed edge disjointly, since any embedding of the Cayley sum-graph of $H$ over $U$ into $V(G)$ has edges in common with $G$ --- see the full proof in Section~\ref{sc:proof_upper2}.

\subsection{Proof of Theorem~\ref{th:upper_deterministic}, part 1}
\label{sc:proof_upper1}

The first part of Theorem~\ref{th:upper_deterministic} follows from the following assertion, that we prove in this section.

\begin{theorem}
\label{t912}
Let $3\leq d=d(n)\leq \ln n/\ln\ln n$. There is an absolute constant $C>0$, sequences $t_1=t_1(n)$, $t_2=t_2(n)$, $q=q(n)$, and a Cayley sum-graph $\Gamma=\Gamma(n)$ of $\mathbb{Z}_2^{t_1}\times\mathbb{Z}_q^{t_2}$ such that the following is true for every positive integer $n$.
\begin{itemize}
\item
The size of the generating set of $\Gamma$ is at most 
$C n^{1-2/d} d^2 (\ln n)^{2+4/d}$.
\item
$\Gamma$ contains every 
graph with $n$ vertices and maximum degree at most $d$
as a subgraph.
\end{itemize}
\end{theorem}
Therefore, the Cayley sum-graph $\Gamma(n)$ is {\it universal} for the family of all graphs with $n$ vertices and maximum degree $d$. 

\begin{remark}
When $d$ is a constant, the statement and the proof of Theorem~\ref{t912} can be slightly simplified. In particular, we may set $t_2=0$ and omit the sequence $q(n)$ in the statement.
\end{remark}

\begin{proof}[Proof of Theorem~\ref{t912}.] We first describe the graph $\Gamma$ and then show that it can be realised as a Cayley sum-graph and that it is universal. 

\paragraph{Construction of $\Gamma$.} Recall that an $(m,r,\lambda)$-graph is an
$r$-regular graph on $m$ vertices in which the absolute value of any nontrivial eigenvalue is at most $\lambda$. We need the following
result, which is a special case of \cite[Proposition 4]{AR}.
\begin{lemma}[\cite{AR}]
\label{l921}
There exists an absolute constant $b$ so that for every positive integer $p$, there is a Cayley graph $Z(p)$ of $\mathbb{Z}_2^p$ with a loop at every vertex, which is an $(m=2^p,r,r/2)$-graph, where $r \leq b p=b \log_2 |\mathbb{Z}_2^p|$.
Such a graph can be constructed explicitly.
\end{lemma}
Note that for the groups $\mathbb{Z}_2^p$ the notions of a Cayley graph and a Cayley sum-graph are identical. Therefore $Z(p)$ is also a Cayley sum-graph.

In the construction below, we will several times apply the direct product operation. Let us recall that the direct product $G_1\otimes G_2$ of graphs $G_1,G_2$ is the graph on $V(G_1)\times V(G_2)$, where vertices $(v_1,v_2),(u_1,u_2)$ are adjacent if and only if $v_1\sim v_2$ and $u_1\sim u_2$. Everywhere below, $K_q$ is a $q$-clique with a loop at every vertex. In particular, $G\otimes K_q$ is obtained from a graph $G$ that has a loop at every vertex by replacing every vertex with a copy of $K_q$ and by drawing all edges between two cliques whenever the respective vertices are adjacent in $G$. We will rely on the following fact --- see, e.g.,~\cite{spectra} for a comprehensive survey on spectra of graph products.
\begin{claim}
\label{cl:product_spectra}
If $G_1$ is an $(m,r,\lambda)$-graph, then $G_1\otimes K_q$ is an $(mq,rq,\lambda q)$-graph.
\end{claim}

Let $c$ be a large enough constant. Given a (large) integer $n$, let $m$ be the smallest integer such that there exist a positive integer $p$ and a positive odd integer $q\leq d$ satisfying
$$
 \mu:=\left(\frac{n}{c d \ln^2 n}\right)^{1/d} \leq 2^p\cdot q=m.
$$
Let $Z(p)$ be a Cayley $(2^p,r,r/2)$-graph as in Lemma~\ref{l921}. Thus $r \leq b p$. Let $Z=K_q\otimes Z(p)$. In particular, $Z$ has a loop at every vertex. Due to Claim~\ref{cl:product_spectra}, $Z$ is an $(m,rq,rq/2)$-graph
\begin{enumerate}
\item with $m\leq \mu(1+12/d)$ vertices; 


\item with degree $rq\leq bpq\leq bpd=O(pd)=O(\ln n)$;
\item that can be realised as a Cayley sum-graph of the group $\mathbb{Z}_2^p\times\mathbb{Z}_q$.
\end{enumerate}
To prove Property 1, let us fix the smallest integer $p'$ such that $2^{p'}d\geq\mu$. Then, let us fix the smallest integer $q'$ such that $2^{p'}q'\geq\mu$. Since $d\leq\ln n/\ln\ln n$ and $\mu \geq (1-o(1))\ln n$, we get that $p\geq p'>0$. Therefore, $\mu\leq m\leq 2^{p'}q'$. Moreover, since $2^{p'-1}d<\mu$, we get that $q'>d/2$ and so, $2^{p'}q'=2^{p'}(q'-1)\frac{q'}{q'-1}<\mu\frac{q'}{q'-1}<\mu\frac{d}{d-2}$. Therefore, $m<\mu(1+2/(d-2))<\mu(1+12/d)$, as needed. Property 3 follows from the fact that every vertex of $Z$ is a vector $\{u,v\}$ with $v\in\mathbb{Z}_2^p$, where $\{u_1,v_1\},\{u_2,v_2\}$ are adjacent iff $v_1+v_2$ belongs to the generating set of $Z(p)$.

Let $F=F(n)$ be the graph with $V(F)=Z^d$, where two vertices $(x_1, \ldots ,x_d)$ and $(y_1, \ldots ,y_d)$ are adjacent iff there are two indices $1 \leq i<j \leq d$ so that $\{x_i,y_i\}$ and $\{x_j,y_j\}$ are edges of $Z$. Due to the first property of $Z$, the graph $F$ has $m^d=\Theta(n/(d\ln^2 n))$ vertices and, due to the second property of $Z$, the graph $F$ is regular of degree at most $m^{d-2} (rq)^2 {d \choose 2}$. Note also that it has a loop
in each vertex, as so is the case with $Z$. Finally, the desired graph $\Gamma$ is defined as $K_{2^s}\otimes F$, where $2^s$ is the smallest power of $2$ larger than $1.1c d \ln^2 n$.

\paragraph{Representing $\Gamma$ as a Cayley sum-graph.} Let us first prove that $\Gamma$ is a Cayley sum-graph of $\mathbb{Z}_2^{t_1}\times\mathbb{Z}_q^{t_2}$, whose generating set has size at most $C n^{1-2/d} d^2 (\ln n)^{2+4/d}$. It is not difficult to see that we may take $t_1=pd+s$ and $t_2=d$. Indeed, each vertex of $\Gamma$ is a vector
consisting of $d+1$ blocks:
$(x_1,x_2,\ldots ,x_d, w)$. Here each $x_i\in\mathbb{Z}_2^p\times\mathbb{Z}_q$ is of length $p+1$ and $w$ is a binary vector of length $s$. Two vertices $(x_1,x_2,\ldots ,x_d, w)$
and $(x'_1,x'_2,\ldots ,x'_d, w')$ are adjacent
iff there are two indices
$1 \leq i<j \leq d$ so that 
$x_i+x'_i$ and $x_j + x'_j$ are in the generating
set of the graph $Z$ (and the sums in all the other blocks are arbitrary). Moreover, every vertex in $\Gamma$ has degree at most 
$$
 2^s m^{d-2}r^2q^2 {d \choose 2}=O\left(n^{1-2/d} d^{2+2/d} (\ln n)^{2+4/d}\right)=O\left(n^{1-2/d} d^2 (\ln n)^{2+4/d}\right),
$$
as needed. 

\paragraph{Universality of $\Gamma$.} The following lemma completes the proof of Theorem~\ref{t912}.
\begin{lemma}
\label{l931}
Let $G$ be an arbitrary graph with $n$ vertices and maximum degree at most $d$. Then there is a
homomorphism $f: V(G) \mapsto V(F)$ from $G$ to $F=F(n)$ so that for, every vertex $v \in V(F)$,
\begin{equation}
\label{e31}
|f^{-1}(v)| \leq 1.1 c d \ln^2 n.
\end{equation}
\end{lemma}
Indeed, take any graph $G$ on $n$ vertices and with maximum degree at most $d$ and consider a homomorphism $f: V(G) \mapsto V(F)$ as in Lemma~\ref{l931}. Then, we may define an injective homomorphism $f':V(G)\mapsto\Gamma$ from $G$ to $\Gamma=\Gamma(n)$ in the following way. For every $v\in V(F)$ such that $f^{-1}(v)\neq\emptyset$, take a set $X_v$ of size $f^{-1}(v)$ from the $2^s$-clique of $\Gamma=K_{2^s}\otimes F$ containing all the vertices $\{\cdot,v\}$ and set $f'(f^{-1}(v))=X_v$.

The proof of Theorem~\ref{t912} is completed using Lemma \ref{l931}. In what follows we prove this lemma.
\end{proof}

\begin{proof}[Proof of Lemma~\ref{l931}.]
While the proof of Lemma~\ref{l931} closely follows the argument in \cite[Theorem~4.1]{AC}, we include the full details here for completeness.

The proof of this lemma starts with the decomposition result \cite[Theorem 3.1]{AC} that shows that $G$ can be covered by $d$ spanning subgraphs $G_1,\ldots,G_d$, so that every edge is covered precisely twice, and each of $G_i$ can be mapped homomorphically to a path $P_n$ with a loop at each vertex,  such that the inverse image of any vertex of $P_n$ consists of at most 4 vertices. For $i\in[d]$, let $g_i$ be such a homomorphism from $G_i$ to $P_n$.

Let $f_1,\ldots,f_d$ be independent random walks on $Z$. They can be viewed as random embeddings of $P_n=(v_1,\ldots,v_n)$ into $Z$: first, $f_i(v_1)$ is a uniformly random vertex of $Z$. Then, for every $j\in\{2,\ldots,n\}$, $f_i(v_j)$ is a uniformly random neighbour of $f_i(v_{j-1})$. The desired homomorphism from $G$ to $F$ is defined as follows: $f(v)=(f_1(g_1(v)),\ldots,f_d(g_d(v)))$. Let us observe that $f$ is indeed a homomorphism: if $\{u,v\}\in E(G)$, then there exist $1\leq i<j\leq d$ such that $\{u,v\}\in E(G_i)\cap E(G_j)$. Therefore, $\{f_i(g_i(u)),f_i(g_i(v))\}\in E(Z)$ and $\{f_j(g_j(u)),f_j(g_j(v))\}\in E(Z)$. Due to the definition of $F$, we get that $\{f(u),f(v)\}\in E(F)$. It only remains to show that $f$ satisfies~\eqref{e31} for every $v\in V(F)$, with positive probability.

Fix a vertex $v=(x_1,\ldots,x_d)\in V(F)$ and let $a>0$ be a large enough constant. Let us say that $U\subset V(G)$ satisfies the {\it distance requirement}, if, for every two vertices $u,u'\in U$ and every $i\in[d]$, the distance between $g_i(u)$ and $g_i(u')$ in $P_n$ is at least $a\ln n$. Due to the Hajnal-Szemer\'{e}di Theorem~\cite{HS}, that asserts that the set of vertices of any graph with
maximum degree smaller than $\Delta$ can be partitioned into $\Delta$ independent sets of equal size, we get that $V(G)$ can be partitioned into $8ad\ln n$ sets of equal size satisfying the distance requirement. Indeed, it follows by applying this theorem to the graph on $V(H)$, in which two vertices $u,u'$ are adjacent iff there exists an $i\in[d]$ such that the distance between $g_i(u)$ and $g_i(u')$ in $P_n$ is smaller than $a \ln n$. It is easy to see that the maximum degree of this auxiliary graph is less than $8ad\ln n$.

It then remains to prove that, for every set $U$ from the partition and every $v\in V(F)$, the number of vertices from $U$ that $f$ maps to $v$ is at most 
$$
 \frac{1.1|U|}{m^d}=1.1\frac{n}{8a d\ln n\cdot m^d}\leq\frac{1.1\cdot c\ln n}{8a}
$$ 
with probability at least $1-1/n^2$, and then apply the union bound over all $8ad\ln n=o(n)$ sets $U$ in the partition and all $m^d=o(n)$ vertices $v\in V(F)$. 

For $j\in[d]$, let $U_j$ be the set of vertices $u\in U$ such that $f_i(g_i(u))=x_i$ for all $i\in[j]$ and let $U_0=U$. We are going to prove by induction that, for every $j\in[d]$, $|U_j|$ is stochastically dominated by $\mathrm{Bin}(|U|,(1/m+1/n^4)^j)$. The base of induction $j=0$ is trivial. For the induction step $(0,\ldots,j-1)\to j$, we order all the vertices $u_1,\ldots,u_r$ of $U_{j-1}$, according to the position of their images under $g_j$ on $P_n$. We recall that $f_j$ is a random walk on $Z$. For every $s\in[r]$, assuming that the values of $f_j(g_j(u_1)),\ldots,f_j(g_j(u_{s-1}))$ have been already explored, we note that, due to the distance restriction, the walk performs at least $4\ln n$ steps in order to reach the ``time step'' $g_j(u_s)\in P_n$ after the ``time step'' $g_j(u_{s-1})$. Since $Z$ is an $(m,rq,rq/2)$-graph, it is rapidly mixing (see, e.g.,~\cite[Proposition 3]{DS}): in particular, for large enough $a$, $a\ln n$ steps is enough to guarantee that 
\begin{multline*}
 \mathbb{P}(f_j(g_j(u_s))=x_j\mid f_j(g_j(u_1))=\ldots=f_j(g_j(u_{s-1}))=x_j)\\
 =\mathbb{P}(f_j(g_j(u_s))=x_j\mid f_j(g_j(u_{s-1}))=x_j)\leq 1/m + 1/n^4.
\end{multline*}
Therefore, $|U_j|=|\{s\in[r]:\,f_j(g_j(u_s))=x_j\}|$ is stochastically dominated by $\mathrm{Bin}(|U_{j-1}|,1/m+1/n^4)$, which is, by the induction assumption, stochastically dominated by $\mathrm{Bin}(|U|,(1/m+1/n^4)^j)$.

This concludes the proof. Indeed, since $(1/m+1/n^4)^d\leq(1/m)^d \exp(md/n^4)=(1/m)^d+o(1/n^4)$, by the Chernoff bound~\cite[Theorem 2.1]{JLR} and due to the bound $m\leq\mu(1+12/d)$, 
$$
 \mathbb{P}(|U_d|>1.1|U|/m^d)\leq\exp(-|U|/(300 m^d))\leq\exp(-c\ln n/(e^{12}\cdot 2400a))=o(1/n^2),
$$
if $c>e^{21}\cdot a$, say.

\end{proof}

\subsection{Proof of Theorem~\ref{th:upper_deterministic}, part 2}
\label{sc:proof_upper2}

Let us prove the following stronger statement: For every abelian group $(H,+)$ of order $n$, every graph $G$ on $[n]$ with maximum degree $d$, and {\bf every} subset $U$ of $H$ of size $\lceil n/(2d)\rceil -1$, there exists a bijection $f$ from $[n]$ to $H$ so that
for every edge $\{u,v\}\in E(G)$, the sum $f(u)+f(v)$ is not in $U$.

Fix an abelian group $H$ of order $n$, a graph $G$ on $[n]$ with maximum degree $d$, and a subset $U$ of $H$ of size $\lceil n/(2d)\rceil -1$. Let $G'$ be the graph obtained from the Cayley sum-graph of $H$, in which $h_1$, $h_2$ are adjacent
iff $h_1+h_2\in U$, by erasing all loops. The maximum degree of $G'$ is clearly at most
$|U| < n/(2d)$. Since twice the product of the maximum degree of $G'$ and the maximum degree of $G$ is smaller than $n$, the complement of $G'$ contains an isomorphic copy of $G$ --- see~\cite[Theorem 4.2]{Catlin},~\cite[Theorem 3]{SauerSpencer}. According to the definition of $G'$, this embedded copy of $G$ does not have any sum in $U$, as required. This completes the proof.

\section{Discussions}
\label{sec:discussions}

The construction of an expander with a logarithmic sum-set presented in~\cite{AABL} is a Cayley sum-graph, where the degree is chosen sufficiently large to ensure the desired expansion properties. This approach does not yield sparse expanders with small sum-sets. Random graphs with the same degrees have better expansion properties and less symmetries, which naturally leads to bigger sum-sets. We suspect that Ramanujan graphs behave similarly to random regular graphs, i.e. that they have sum-sets of polynomial size in 
$n$ over any abelian group. More generally, for a fixed constant $d$, it would be interesting to understand how ${\sf S}(G)$ depends on the spectral gap of a $d$-regular expander graph $G$.

The upper and lower bounds provided by Theorems~\ref{th:upper_deterministic}~and~\ref{th:lower_main}, respectively, exhibit different asymptotic behaviours when $d=O(\log^2n)$. It remains an open question whether random $d$-regular graphs attain the asymptotically maximum value of ${\sf S}(G)$ within the class  of graphs with maximum degree $d$ in this regime. We are particularly interested in determining the correct asymptotic behaviour for: (1) the minimum sum-sets of random regular graphs, and (2) the maximum ${\sf S}(G)$ among graphs with maximum degree $d$.

Finally, we note that the proof strategy of Theorem~\ref{th:lower_main} in the case $d>\sqrt{n}$ can be applied for any $d\gg\log n$ since in this range the coupling of $\mathcal{G}_{n,d}$ with the `top layer' of the `sandwich' is known. It actually implies a slightly better lower bound for all $d=\exp(\omega(\log\log n))$: whp ${\sf S}(\mathbf{G}_n)=n\left(1-O\left(\frac{\ln^2 n\ln(1+\ln n/\ln d)}{d\ln d}\right)\right)$. In particular, whp ${\sf S}(\mathbf{G}_n)=n(1-O(\ln n/d))$ when $d=n^{\Theta(1)}$, making the second order terms in the lower and upper bounds differ by a $O(\log n)$-factor.

\section*{Acknowledgements}
The authors are grateful to Elad Tzalik for useful discussions. Part of this research was performed during a visit of the first author as a Clay Senior Scholar at the IAS/Park City Mathematics Institute in July, 2025.

\end{document}